\documentclass[reqno,11pt]{amsart}
\usepackage[margin=1in]{geometry}
\title[Directional discrepancy: interval of rotations]{Lower bounds for the directional discrepancy with respect to an interval of rotations}
\date{\today}
\author[D. Bilyk]{Dmitriy Bilyk}
\address{School of Mathematics, University of Minnesota, Minneapolis, MN 55455} 
\email{dbilyk@math.umn.edu, michmast@math.umn.edu}
\author[M. Mastrianni]{Michelle Mastrianni}
\subjclass[2000]{Primary 11K38, 42B05; Secondary 52A10}
\thanks{Both authors have been supported by the National Science Foundation (grant DMS 2054606).} 

\keywords{discrepancy, average decay of  the Fourier transform}

\usepackage{amsmath,amssymb, amsfonts, dsfont}
\usepackage{todonotes}
\usepackage{enumerate}

\usepackage{booktabs}
\usepackage{algcompatible}
\usepackage{algorithm}
\usepackage{algpseudocode}
\usepackage{hyperref}
\usepackage{verbatim}
\usepackage{mathtools}
\usepackage{float}
\usepackage{graphicx}
\usepackage{subcaption}

\newtheorem{theorem}{Theorem}[section]
\newtheorem{lemma}[theorem]{Lemma}

\newtheorem{thm}[theorem]{Theorem}
\newtheorem{corollary}[theorem]{Corollary}
\newtheorem{proposition}[theorem]{Proposition}

\theoremstyle{definition}
\newtheorem{definition}[theorem]{Definition}

\newtheorem{remark}[theorem]{Remark}

\numberwithin{equation}{section}

\newcommand{\fatone}{\mathds{1}}

\newcommand{\area}{\text{area}}

\begin{document}
\maketitle
\begin{abstract}
We  show that the lower bound for the optimal directional discrepancy with respect to the class of rectangles in $\mathbb{R}^2$ rotated in a restricted interval of directions $[-\theta,\theta]$ with $\theta < \frac{\pi}{4}$  is of the order at least $N^{1/5}$ with a constant depending on $\theta$. 
\end{abstract}

\section{introduction}

In the present paper we discuss the directional discrepancy in the plane. Consider a set  $\Omega \subset \left[ -\frac{\pi}{4},\frac{\pi}{4}\right]$,  which we  shall call the \emph{allowed rotation set}. 
Let  the class of sets $\mathcal{R}_{\Omega}$ contain all rectangles in $\mathbb{R}^2$ which make an angle of $\omega$ with the $x$-axis, where $\omega \in \Omega$:
\begin{equation}\label{eq:defROmega}
\mathcal{R}_{\Omega} = \{\text{rectangles } R: \text{ one side of } R \text{ makes angle } \omega \in \Omega \text{ with the x-axis}.  \}
\end{equation}
We then define the directional discrepancy of an $N$-point set $P \subset [0,1]^2$ in the directions given by $\Omega$ as the extremal discrepancy of $P$  with respect to elements of $\mathcal R_\Omega$. In other words, 
\begin{equation}\label{eq:defdiscr}
 D (P,  \mathcal{R}_{\Omega} ) = \sup_{R \in \mathcal{R}_{\Omega}} \big| |P \cap R | -  N \cdot \textup{vol} (R \cap [0,1]^2 ) \big| ,
 \end{equation}
and the optimal \textit{directional discrepancy} is 
\begin{equation}\label{eq:defdiscrN}
D (N, \mathcal{R}_{\Omega} ) = \inf_{|P|= N}  D ( P, \mathcal{R}_{\Omega} ).
\end{equation} 
The two extreme cases of the directional discrepancy are well studied and yield very different behavior of optimal asymptotic estimates.
\begin{itemize}
\item {\bf{No rotations. Axis-parallel rectangles.}}  When $\Omega = \{0\}$ is a singleton  (i.e. all rectangles point in a fixed direction), one recovers the classical case of axis-parallel rectangles, which results in  very small,  logarithmic discrepancy \cite{Schmidt,Lerch}:
\begin{equation}\label{eq:ap}
D (N, \{ 0\} ) \approx \log N. 
\end{equation}
However more complex classes of test set yield much larger discrepancy estimates:
\item {\bf{All directions. Arbitrarily rotated rectangles.}} On the other hand, the case of all directions, i.e.  $\Omega = \left[-\frac{\pi}{4},\frac{\pi}{4}\right] $ (that is, when $\mathcal{R}_{\Omega}$ consists of all arbitrarily rotated rectangles) is also fairly well-understood: we have polynomial discrepancy with bounds
\begin{equation}\label{eq:Beck}
N^{1/4} \lesssim \mathcal{D}(N,\mathcal{R}_{[-\frac{\pi}{4},\frac{\pi}{4}]}) \lesssim N^{1/4}\sqrt{\log N}.
\end{equation}
Both the lower and upper bound are due to Beck \cite{Beck87}; in \S \ref{sec:allrot} we outline Beck's proof of the lower bound in the case of all rotations. Perhaps unsurprisingly,  discrepancy estimates are almost identical in the case of discs in $\mathbb R^2$ \cite{Montgomery}.
\end{itemize}

Our main goal in this paper is to understand what happens ``in between'' these two extreme cases (a single direction and all possible directions). The broader question is about the interplay between the geometry of the set $\Omega$ of allowed rotations and the sharp discrepancy bounds with respect to $\mathcal{R}_{\Omega}$ (or, even more generally, the influence of the  geometric properties of the class of test sets on discrepancy). 

Various different choices of $\Omega$, as well as their geometric characteristics, are of interest, such as Cantor-type sets, different infinite sequences of directions, and sets with given (Hausdorff or Minkowski) dimension. Upper bounds for some of these cases have been studied in \cite{BMPS11,BMPS16}, where the arguments rely heavily on the Diophantine properties of the allowed rotation set $\Omega$.

However, we are not aware  of any results about lower bounds for these intermediate cases in the literature. Already the simplest case when $\Omega$ is an \emph{interval} of directions is not easy to understand and produces many open questions. In this paper we prove the first  result for this case: namely, a lower bound of the order $N^{1/5}$. Our main result, proven in \S \ref{sec:proof}, is stated more precisely below.

\begin{thm}
\label{thm:restrictedinterval}
Let $\Omega = [-\theta, \theta]$ with $\theta < \frac{\pi}{4}$, i.e. $\Omega$ is a restricted interval of rotations. Then there is a constant $\gamma >0 $ such that for all  $N \ge \gamma \theta^{-5}$,
$$D(N, \mathcal{R}_{\Omega}) \ge c N^{1/5}$$
for some constant $c$ independent of $\theta$.
\end{thm}

One can easily restate this result as a bound for all $N\in \mathbb N$, but with a constant that depends on the length of the interval.

\begin{corollary}\label{cor:thm}
Let $\Omega = [-\theta, \theta]$ with $\theta < \frac{\pi}{4}$. There exists a constant $c'>0$ such that for all  $N\in \mathbb N$,
$$D(N, \mathcal{R}_{\Omega}) \ge c' \theta N^{1/5}.$$
\end{corollary}
\noindent It is natural, in accordance with \eqref{eq:ap}, that the constant decays as $\theta$ goes to zero.

Part of the proof of these results is based on and inspired by  the arguments recently developed by Brandolini and Travaglini  in \cite{BT} for discrepancy estimates with respect to dilations and translations of a convex body with given local  convexity and smoothness properties.

It remains unclear whether or not it is possible to improve this lower bound for directional discrepancy on restricted intervals. In particular, one might expect that it should have order $N^{\frac{1}{4}}$, just as in the case of all rotations, but this question is wide open.

%

We proceed in \S \ref{sec:allrotations} with an exposition of the Fourier transform decay estimates, which  are used, in particular, in the proof of the $\Omega(N^{1/4})$ lower bound \eqref{eq:Beck} for the  discrepancy in the extreme case of all rotations, see \S\ref{sec:allrot}. We adapt the Fourier bounds to our problem, see Proposition \ref{p:fourier}, and explain  why the proof technique for all rotations does not generalize easily to the case of a restricted interval. In \S\ref{s:BT} we outline some relevant   results and ideas from \cite{BT}  and then  give a proof of Theorem \ref{thm:restrictedinterval} in \S\ref{sec:proof}. 

Throughout the paper, we shall use the notation $A \lesssim B$ meaning that there exists an absolute constant $C$ such that $A\le CB$ (the implicit constant is assumed to be independent of $N$ and $\theta$ and other relevant parameters; in particular, we shall carefully trace the dependence on $\theta$). The relation $A\approx B$ means that $A\lesssim B$ and $B \lesssim A$. The constants $c$, $c'$, $c_i$ etc. are not necessarily the same from line to line.

\begin{remark}\label{rem:discr}
The proof of Theorem \ref{thm:restrictedinterval} presented in \S \ref{sec:restrictedintervals} employs Fourier series, and therefore, treats the unit square $[0,1]^2$ as the torus $\mathbb T^2$, and  from that viewpoint, rectangles in the class $\mathcal{R}_{\Omega}$ are ``periodic'' rectangles. This gives rise to a slightly different notion of discrepancy than the one defined in \eqref{eq:defdiscr}, where one looks at the intersection of rectangles with the unit square rather than  periodic extensions. It is easy to see, however, that the lower bound for the discrepancy with respect to  periodic rectangles immediately yields a lower bound for the former discrepancy \eqref{eq:defdiscr}. This can be seen from Figure \ref{fig:rotatedrects1}. Indeed, if  a periodic rectangle has  large discrepancy, then at least one of at most four intersections with the unit square also has large discrepancy, possibly with different constant.


 \begin{figure}[H]
\centering
\includegraphics[scale=0.4]{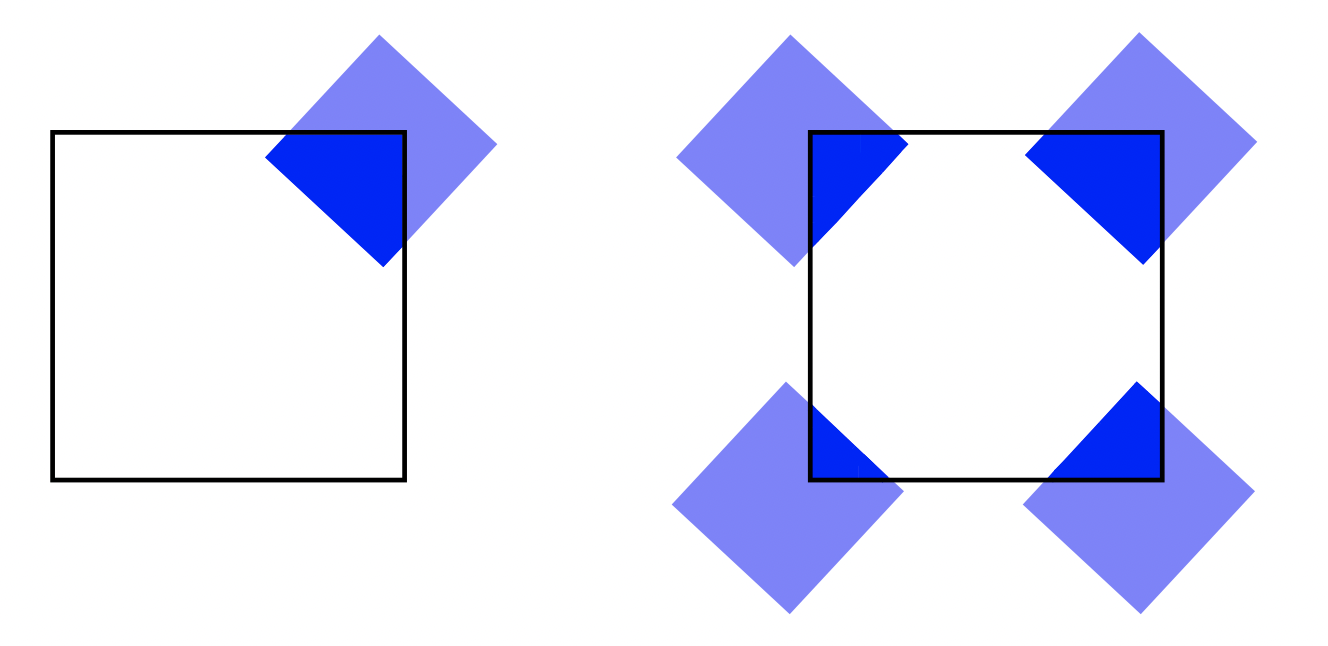}
\caption{On the left, a rotated rectangle intersected with the unit square; on the right, the same rectangle viewed as periodic.}
\label{fig:rotatedrects1}
\end{figure}

Finally, we remark that the lower bounds we obtain are for the $L^2$ discrepancy, i.e., appropriate quadratic averages of the discrepancy over rectangles in $\mathcal{R}_{\Omega}$, which naturally implies lower bounds for the supremum (extremal) discrepancy as defined in \eqref{eq:defdiscr}--\eqref{eq:defdiscrN}.
\end{remark}

\section{Fourier transform techniques}
\label{sec:allrotations}

Our proof will rely on the average decay estimates for the Fourier transform of the test sets. This technique is often employed in various geometric problems (including discrepancy) in which rotational invariance or curvature is present, see e.g. \cite{Beck87,BCIT,BT,Ios}. In particular, the Fourier transform method is used in Beck's  proof  \cite{Beck87} of the  lower bound for the  discrepancy with respect to rectangles rotated in arbitrary directions \eqref{eq:Beck}, which is beautifully presented in \cite[Sec. 7.1]{Matousek} and which we briefly explain in \S\ref{sec:allrot}. 

We now present the average  Fourier decay bounds relevant to our problem.  Instead of rectangles $\mathcal{R}_{[-\theta, \theta]}$, we shall restrict our attention just to the class of  {\it{rotated squares}} $\mathcal{S}_{[-\theta, \theta]}$.  Let $\fatone_{r,\nu}$ denote the indicator function of  the square centered at the origin with ``radius'' $r$ and making angle $\nu$ with the $x$-axis, i.e. the square $[-r,r]^2$ rotated by the angle $\nu$ counterclockwise.  We shall be interested in the behavior of the following quantity: 
\begin{equation}\label{eq:phi}
\varphi_{R,\theta}(\xi) =    \frac{1}{R {\theta}} \int_{R/2}^R  \int_{-\theta}^{\theta} |\widehat{\fatone}_{r,\nu}(\xi)|^2  \hspace{0.1cm} d\nu \hspace{0.1cm} dr, \,\,\,\,\,\,\, \xi  = (\xi_1, \xi_2)\in \mathbb R^2, 
\end{equation}
where $R>0$ and $\theta \in (0,\pi/4]$, i.e. this is  the squared  Fourier transform of the indicator of the square averaged over rotations in the interval $[-\theta, \theta]$ and dilations between $R/2$ and $R$. The relevance of this quantity to discrepancy estimates will become apparent in, e.g., \S\ref{sec:allrot} and \S\ref{sec:proof}.

For a fixed direction $\nu =0$, it is easy to compute
\begin{equation}\label{eq:nu0}\widehat{\fatone}_{r,0}(\xi) =\Big(\int_{-r}^r e^{-2\pi i x_1 \xi_1} dx_1\Big)\Big(\int_{-r}^r e^{-2\pi i x_2 \xi_2} dx_2 \Big)  =   \frac{\sin(2\pi \xi_1 r)}{\pi \xi_1} \cdot  \frac{\sin(2\pi \xi_1 r)}{\pi \xi_1}.
\end{equation}
Zeros of this function explain the need for averaging over dilations: it ``smears'' the function, thus eliminating the zeros and allowing one to obtain uniform lower bounds for $\varphi_{R,\theta}(\xi)$. 

We have the following lower bounds for the average Fourier decay $\varphi_{R,\theta}(\xi)$:
\begin{proposition}\label{p:fourier}
Let $\theta \in (0, \pi/4)$ and $R>0$. 
\begin{enumerate}[(i)]
\item\label{t1} There exists a constant $c'>0$ such that 
\begin{equation}\label{eq:f1}
\varphi_{R,\theta}(\xi) \gtrsim  \frac{R}{\theta |\xi|^3}
\end{equation}
  {whenever $\operatorname{arg} (\xi) \in ( -\theta/2, \theta/2)$ and $| \xi | \ge \frac{c'}{\theta R}$. }
\item\label{t2} There exists a constant $c''>0$ such that 
\begin{equation}\label{eq:f2}
\varphi_{R,\theta}(\xi) \gtrsim  \frac{1}{ |\xi|^4}
\end{equation}
  {for all $\xi \in \mathbb R^2$ with $| \xi | \ge \frac{c''}{ R}$. }
\end{enumerate}
\end{proposition}

Before we proceed to the proof of this proposition we would like to make a couple  remarks:

\begin{remark} In the case when $\theta = \pi/4$, i.e. if we consider squares rotated in {\it{all}} directions, the decay estimate \eqref{eq:f1} of part \eqref{t1} holds independently of $\operatorname{arg} (\xi)$, i.e. for all $\xi$ large enough. Precisely this estimate allows one to obtain the lower bound \eqref{eq:Beck} for the discrepancy of arbitrarily rotated rectangles, see the discussion in \S\ref{sec:allrot}.  When $\theta < \pi/4$, however, i.e. in the case of restricted intervals, this estimate holds only on a sector with aperture slightly smaller than the interval of rotations, which prevents one from obtaining the same discrepancy bound  and leads to much more delicate arguments. 
\end{remark}

\begin{remark}
Estimate \eqref{eq:f2} of part \eqref{t2} of the proposition actually  holds for each square, without the need to average over rotations. In fact, even more generally, such an estimate holds for any planar convex set, see e.g. Theorem 24 in \cite{BT} as well as the discussion in \S\ref{s:BT}.  Nevertheless, we shall present a simple proof for the square below. 
\end{remark}

\begin{proof}[Proof of Proposition \ref{p:fourier}] We start  with part \eqref{t1}. The beginning of the argument, as well as the notation, closely follows Lemma 7.5 in \cite{Matousek}.   As we already computed  in \eqref{eq:nu0},
\begin{equation}
\label{eqn:indicatorcalc}
  |\widehat{\fatone}_{r,0}(\xi)|^2 =  \frac{\sin^2(2\pi \xi_1 r)\sin^2(2\pi \xi_2 r)}{ {\pi^4} \xi_1^2 \xi_2^2}.
\end{equation}
Note that $|\widehat{\fatone}_{r,\nu}(\xi)|^2 = |\widehat{\fatone}_{r,0}(\widetilde{\xi})|^2$ where $\widetilde{\xi} = (\widetilde{\xi}_1, \widetilde{\xi}_2) = (\xi_1 \cos \nu + \xi_2 \sin \nu, -\xi_1  \sin \nu + \xi_2 \cos \nu).$  For the moment,  we will assume $\xi = (|\xi|, 0)$. In this case, 
\begin{align*}
|\widehat{\fatone}_{r,\nu}(\xi)|^2 = \left(\frac{\sin ( 2\pi |\xi| r \cos \nu)}{ \pi  |\xi| \cos \nu}\right)^2 \, \left(\frac{\sin ( 2\pi |\xi| r \sin \nu)}{\pi |\xi| \sin \nu}\right)^2.
\end{align*}
Note that for $\nu \in [-\theta,\theta]$,  $|\sin \nu| \approx |\nu|$ and $\cos \nu \approx 1$. Hence, restricting to the interval $(0,\theta)$ due to  symmetry, we can write 
\begin{equation}
\label{eqn:varphiequation}
\varphi_{R,\theta}(\xi) \approx   \frac{2}{R} \int_{R/2}^R \frac{1}{\theta} \int_{0}^{\theta}  \frac{\sin^2( 2\pi |\xi| r \cos \nu)\sin^2(2 \pi |\xi| r \sin \nu)}{|\xi|^4 \nu^2} \hspace{0.1cm} d\nu \hspace{0.1cm} dr
\end{equation}

Assume  that $\theta > \frac{1}{R|\xi|}$, i.e. $|\xi| > \frac{1}{R\theta}$. Consider the integral in \eqref{eqn:varphiequation} with  the integration in $\nu$ restricted  to the interval from $\frac{1}{R|\xi|}$ to $\theta$:
\begin{align}
\label{eqn:smallintegralestimate}
\begin{split}
\frac{2}{R} \int_{R/2}^R  \frac{1}{\theta} \int_{\frac{1}{R|\xi|}}^{\theta} &  \frac{\sin^2(2\pi |\xi| r \cos \nu) \sin^2(2\pi |\xi|r\sin \nu)}{|\xi|^4 \nu^2}  d\nu dr \lesssim \frac{2}{R} \int_{R/2}^R  \frac{1}{\theta} \int_{\frac{1}{R|\xi|}}^{\theta}  \frac{1}{|\xi|^4 \nu^2} d\nu dr\\
& =  \frac{2}{R} \int_{R/2}^R  \frac{1}{\theta} \Big[-\frac{1}{\theta|\xi|^4} + \frac{R}{|\xi|^3}\Big] dr  \lesssim  \frac{R}{\theta |\xi|^3}.
\end{split}
\end{align}
Now let us consider the integration interval $(0, \frac{1}{R|\xi|})$. On this interval, we have 
$\sin(2\pi r|\xi| \sin \nu) \approx R|\xi| \nu$.
In addition, for any $\nu$, we have 
\begin{align}
\label{eqn:smallestimate2}
\frac{2}{R} \int_{R/2}^R \sin^2( 2\pi r|\xi| \cos \nu) dr \approx 1.
\end{align}
The upper bound is clear, while the  lower bound follows from the averaging over $r$. Indeed, upon  rescaling and setting $C= \pi |\xi| R \cos \nu$, one obtains the integral 
\begin{equation}
\label{eqn:sinefcn}
 \int_{1}^2 \sin^2( Ct)\, dt = \frac{1}{2} - \frac{1}{2} \int_{1}^2 \cos( 2Ct)\,  dt \ge \frac12 - \frac{1}{C} \ge \frac14, 
\end{equation}
if $|\xi| R $ is large enough. \\

We thus have 
\begin{align*}
  \frac{2}{R\theta } & \int_{R/2}^R   \int_0^{\frac{1}{R|\xi|}}  \frac{\sin^2( 2\pi |\xi| r \cos \nu)\sin^2(2\pi |\xi| r \sin \nu)}{|\xi|^4 \nu^2}  d\nu dr  
   \approx  \frac{2}{R\theta } \int_{R/2}^R   \int_0^{\frac{1}{R|\xi|}}  \frac{  \sin^2(2\pi |\xi| r \cos \nu) R^2 |\xi|^2 \nu^2}{|\xi|^4 \nu^2}  d\nu dr \\
& =  \frac{1}{\theta}  \int_0^{\frac{1}{R|\xi|}}  \frac{R^2 |\xi|^2 \nu^2}{|\xi|^4 \nu^2} \Big(\frac{2}{R} \int_{R/2}^R  \sin^2( |\xi| r \cos \nu) dr \Big)  d\nu \approx  \frac{1}{\theta}  \int_0^{\frac{1}{R|\xi|}}  \frac{R^2}{|\xi|^2}  d\nu  \approx  \frac{R}{\theta |\xi|^3}.
\end{align*}
Hence, putting this together with \eqref{eqn:smallintegralestimate}, we have  
\begin{equation}\label{eq:phi1}
\varphi_{R,\theta}(\xi)  \approx  \frac{R}{\theta |\xi|^3}.
\end{equation} 
This proves  estimate \eqref{eq:f1} in the specific case $\xi = (|\xi|, 0)$. We now show how to remove this assumption. 

Let $\xi = (\xi_1, \xi_2)$ with $\theta_0 = \operatorname{arg} (\xi) = \tan^{-1}(\frac{\xi_2}{\xi_1})  = \alpha \in (0, \frac{\theta}{2})$, and denote $\xi' = (|\xi|,0)$ where $|\xi|$ is the length of $\xi$. Then, obviously,  $\widehat{\fatone}_{r,\nu}(\xi) = \widehat{\fatone}_{r,\nu- \alpha}(\xi')$. Therefore, 
\begin{align*} \frac{1}{\theta} \int_0^{\theta} | \widehat{\fatone}_{r,\nu}(\xi) |^2 d\nu =  \frac{1}{\theta} \int_0^{\theta} | \widehat{\fatone}_{r,\nu- \alpha}(\xi') |^2 d\nu  = \frac{1}{\theta} \int_{-\alpha}^{\theta-\alpha} | \widehat{\fatone}_{r,\nu}(\xi') |^2 d\nu  \ge \frac12 \cdot   \frac{2}{\theta} \int_0^{\theta/2} | \widehat{\fatone}_{r,\nu}(\xi') |^2 d\nu.
\end{align*} 
Hence, $$ \varphi_{R,\theta}(\xi) \gtrsim \varphi_{R,\theta/2}(\xi') \gtrsim \frac{R}{\theta |\xi|^3}$$ when $|\xi | \gtrsim \frac{1}{R\theta}$, which finishes the proof of part \eqref{t1} of Proposition \ref{p:fourier}. \\

We now turn to the proof of part \eqref{t2}. Observe that it is enough to prove the bound \eqref{eq:f2} for one direction, e.g. $\nu =0$, without restrictions on $\operatorname{arg} (\xi)$.  Recall that  
$ | \widehat{\fatone}_{r,0} (\xi) |^2 $ is given by \eqref{eqn:indicatorcalc}. 
If both $|\xi_1|$, $|\xi_2| \gtrsim \frac{1}{R}$, then 
$$ \frac{2}{R}  \int_{R/2}^R {\sin^2(2\pi \xi_1 r)\sin^2(2\pi \xi_2 r)}  dr \approx 1,$$ which can be shown similarly  to estimate \eqref{eqn:smallestimate2} in the proof of part \eqref{t1}. Therefore, in this case,
\begin{align*}
\frac{2}{R} \int_{R/2}^R  |\widehat{\fatone}_{r,0}(\xi)|^2 \, dr & =  \frac{2}{R}  \int_{R/2}^R \frac{\sin^2(2\pi \xi_1 r)\sin^2(2\pi \xi_2 r)}{ {\pi^4} \xi_1^2 \xi_2^2} dr \gtrsim \frac{1}{\xi_1^2 \xi_2^2} \ge \frac{2}{(\xi_1^2 +\xi_2^2)^2} = \frac{2}{|\xi|^4}. \end{align*}

On the other hand, if  $|\xi_1| \lesssim \frac{1}{R}$, but $|\xi| \ge |\xi_2| \gtrsim \frac{1}{R}$, then  $$ \frac{\sin^2(2\pi \xi_1 r) }{\pi^2  \xi_1^2 } \approx r^2 \approx R^2 ,$$
and once again invoking \eqref{eqn:smallestimate2} we get
\begin{align*}
\frac{2}{R} \int_{R/2}^R  |\widehat{\fatone}_{r,0}(\xi)|^2 \, dr & \approx  \frac{R^2}{\xi_2^2}  \cdot \frac{2}{R}  \int_{R/2}^R {\sin^2(2\pi \xi_2 r)}  dr \approx  \frac{R^2}{ \xi_2^2} \gtrsim \frac{1}{\xi_2^4  } \approx \frac{1}{|\xi|^4}. \end{align*}

\end{proof}

Since in practice we are going to use the decay estimates of Proposition \ref{p:fourier} when $R\approx 1$ is fixed, we state this case as a separate corollary.
\begin{corollary}\label{cor:phir}
Let $\theta \in (0,\pi/4)$ and assume that $R \approx 1$ is fixed, e.g. $R = \frac{1}{16}$. Then
\begin{equation}\label{eq:phir}
\varphi_{R,\theta} (\xi) \gtrsim 
\begin{cases} \frac{1}{\theta |\xi|^{3}} & \textup{ if } -\theta/2 \le \operatorname{arg} (\xi) < \theta/2 \textup{ and } |\xi| \gtrsim \frac1{\theta}, \\ \frac{1}{|\xi |^{4}} & \textup{ if }  |\xi| \gtrsim 1. \end{cases}
\end{equation} 
\end{corollary}

\subsection{The case of all rotations}\label{sec:allrot}
We would like to explain briefly how Fourier estimates akin to \eqref{eq:f1} are used to prove the classical lower bound for the discrepancy with respect to $\mathcal{R}_{[-\frac{\pi}{4},\frac{\pi}{4}]}$, i.e. all rotated rectangles. This would explain the use of the average Fourier decay and demonstrate the difficulties of the restricted case.  The argument is originally due to Beck in \cite{Beck87}.
\begin{thm}
\label{thm:allrotationslowerbound}
For the class of rectangles rotated in all directions $\mathcal{R}_{[-\frac{\pi}{4},\frac{\pi}{4}]}$, we have
\begin{equation}\label{eq:BeckL}
D(N, \mathcal{R}_{[-\frac{\pi}{4},\frac{\pi}{4}]}) \gtrsim N^{1/4}.
\end{equation}
\end{thm}

\begin{proof}
To simplify matters, we prove the bound for the class $\mathcal{S}_{[-\frac{\pi}{4},\frac{\pi}{4}]}$ of rotated squares; since $\mathcal{S}_{[-\frac{\pi}{4},\frac{\pi}{4}]} \subseteq \mathcal{R}_{[-\frac{\pi}{4},\frac{\pi}{4}]}$, the lower bound also follows for rotated rectangles. The proof makes use of an amplification method which exploits the fact that the $L_2$-discrepancy is a convolution of two functions.

Let $P_N$ be any $N$-point set in $[0,1]^2$. If $S(q,r,\nu) \subset [0,1]^2$ is an axis-parallel square centered at $q$ of side length $2r$ and making angle $\nu$ with the $x$-axis, it is easy to see that the local discrepancy of $P_N$ with respect to $S(q,r,\nu)$ can be expressed as a convolution 
\begin{equation}\label{eq:discrL2}
D(P_N, S(q,r,\nu)) =  \int_{\mathbb{R}^2} \fatone_{r,\nu}(q-p) d\mu(p) = (\fatone_{r,\nu} \ast \mu)(q),
\end{equation}
where $\mu$ is the measure  on $R^2$ given by
$\mu = \sum_{i=1}^N \delta_{p_i} - N\lambda$, with $\lambda$ denoting the restriction of the usual Lebesgue measure on $\mathbb R^2$ to $[0,1]^2$.  Then by Plancherel's theorem, the $L^2$ average of the local discrepancies with respect to shifts $q$ satisfies
\begin{align}
\label{eqn:decomposition}
\|D(P_N, R(q,r,\nu)\|_{L^2(dq)}^2 & = \int_{\mathbb{R}^2} D(P_N, S(q,r,\nu))^2 dq \\ \nonumber & = \int_{\mathbb{R}^2} \widehat{(\fatone_{r} \ast \mu)}(\xi)^2 d\xi = \int_{\mathbb{R}^2} |\widehat{\fatone}_{r,\nu}(\xi)|^2 |\widehat{\mu}(\xi)|^2 d\xi,
\end{align}
i.e. on the Fourier side the ``point component''      $\widehat{\mu}$ and the ``shape component''       $\widehat{\fatone}_{r,\nu}$ are completely separated. 

Now, if $r = \frac{1}{2\sqrt{N}}$, then  the local discrepancy $|D(P_N,S(q,r,\nu))| \ge \frac14$, since ``counting part'' is always integer, while the ``area part'' is $\frac14$. So in this case we have the trivial lower bound \begin{equation}
\label{eqn:trivialbound}
\int_{\mathbb{R}^2} D(P_N,S(q,r,\nu))^2 dq \gtrsim 1.
\end{equation}

The amplification idea is as follows. If $|\widehat{\fatone}_{r,\nu}(\xi)|^2$ would grow linearly in $r$ for any value of $\xi$, then by changing $r = \frac{1}{2\sqrt{N}}$ to $r \approx 1$ we could amplify the trivial lower bound for the square of $L_2$-discrepancy to $\Omega(\sqrt{N})$, thus giving a discrepancy lower bound of $\Omega(N^{1/4})$. This linear growth phenomenon does not  hold for fixed $\nu$ and $r$. Indeed, expression \eqref{eqn:indicatorcalc}   for   $|\widehat{\fatone_{r,0}}(\xi)|^2$  shows that zeros and concentration near the axes create a problem. But this can be fixed by  averaging with respect to dilations and translations, i.e. by considering the quantity $\varphi_{R,\pi/4}(\xi)$.

As mentioned earlier, for {\it{all}} rotations \eqref{eq:f1} holds without any restrictions on $\operatorname{arg} (\xi)$, and even more generally one gets 
\begin{equation}
\label{eqn:decayestimates}
\varphi_{R,\pi/4}(\xi) \approx \min \Big(R^4, \frac{R}{|\xi|^3}\Big) \text{ for all } R,
\end{equation}
which does satisfy the linear growth alluded to earlier
\begin{equation}\label{eq:linear}
\varphi_{aR, \theta}(\xi) \ge k a \cdot \varphi_{R,\theta}(\xi).
\end{equation}
This   easily yields the discrepancy lower bound \eqref{eq:BeckL} for all rotations.
Thus, letting $R_0 = \frac{1}{2\sqrt{N}}$ and $R \in [R_0, \frac{1}{2}),$ we have
\begin{align*}
\frac{2}{R} \int_{R/2}^R   \frac{2}{\pi} \int_{-\pi/4}^{\pi/4} \|D(P_N,& R(q,r,\nu)\|_{L^2(dq)}^2 \, d\nu dr = \frac{2}{R} \int_{R/2}^R \frac{2}{\pi} \int_{-\pi/4}^{\pi/4}  \int_{\mathbb{R}^2} |\widehat{\fatone}_{r,\nu}(\xi)|^2 |\widehat{\mu}(\xi)|^2 d\xi d \nu d r \\ 
&  = \int_{\mathbb{R}^2} |\widehat{\mu}(\xi)|^2 \varphi_{R,2\pi}(\xi) d\xi \ge  k \frac{R}{R_0} \int_{\mathbb{R}^2} |\widehat{\mu}(\xi)|^2 \varphi_{R_0,\pi}(\xi) d\xi
\gtrsim R \sqrt{N},
\end{align*}
where the last inequality follows from the trivial discrepancy bound \eqref{eqn:trivialbound}. 
\end{proof}

The fact that \eqref{eqn:decayestimates} or, more preceisely,  \eqref{eq:f1} holds only on a sector for restricted intervals of directions prevents the amplification argument above from working in this setting. In the following section we shall at least partially overcome this difficulty.

\section{Restricted intervals of rotations}
\label{sec:restrictedintervals}
We turn to the case of a restricted interval of rotations. We prove Theorem \ref{thm:restrictedinterval} inspired by techniques from a recent paper of Brandolini and Travaglini in \cite[2022]{BT}.

As is shown in the previous section, 
the favorable decay estimates for $\varphi_{R, \theta}$ only hold on a sector of Fourier space. Thus, if we wanted to prove an analogue of Theorem \ref{thm:allrotationslowerbound} for the restricted interval of rotations $[-\theta,\theta]$ using the same amplification technique, we would need to integrate the analogue of \eqref{eqn:decomposition} on the sector $S_{\theta} = \{\xi = (\xi_1, \xi_2): \tan^{-1}(\frac{\xi_2}{\xi_1}) \in (-\frac{\theta}{2},\frac{\theta}{2})\}$  rather than all of $\mathbb{R}^2$, i.e. consider the expression 
\begin{align}
\label{eqn:sectorintegral}
\int_{S_{\theta}} |\widehat{\fatone}_{r,\nu}(\xi)|^2 |\widehat{\mu}(\xi)|^2 d\xi
\end{align}
But there is a problem: in order to show that the $L_2$ discrepancy is of roughly the same order as the above expression, we would need to show that the point component $|\widehat{\mu}(\xi)|^2$ carries at least a constant fraction of its mass on that sector. (Note that in the proof for all rotations, it was not necessary to know anything about the behavior of  $|\widehat{\mu}(\xi)|^2$.) Since $\mu$ is the sum of $N$ dirac masses at each point in the point set $P_N$, the behavior of $\widehat{\mu}$ varies wildly depending on $P_N$. It is therefore difficult to conclude anything about the relationship between the $L_2$ discrepancy and  \eqref{eqn:sectorintegral}.

\begin{remark}
\label{rmk:decayestimates}
The decay estimates for $\varphi_{R, \theta}$ in the proof of the amplification lemma match those for a disc, for $\xi$ large enough. Classes of discs and classes of all rotated rectangles tend to behave very similarly in discrepancy settings because of their shared rotational invariance and Fourier transform properties. In the restricted setting, the decay estimates agree with those for a disc only on a restricted sector of Fourier space.
\end{remark}

\begin{remark}
\label{rmk:stefan}
Understanding the behavior of $|\widehat{\mu}(\xi)|^2$ is of interest in other problems. For example, in \cite{Steinerberger} it is shown that one of the first $4N$ Fourier frequencies of the measure $\mu$ on the torus is not $0$: in fact, this holds with the same constant $4$ holds on all manifolds, not just the torus. On the torus, one can prove this directly using the following result of Montgomery \cite{Montgomery}: let $p_1, \cdots, p_N$ be $N$ points in $\mathbb{T}^2$. Then 
$$\sum_{\substack{|k_1| \le X_1 \\ |k_2| \le X_2 \\ (k_1,k_2) \neq  (0,0)}} \Big|\sum_{j=1}^N e^{2\pi i m \cdot p_j}\Big|^2 \ge NX_1X_2 - N^2$$
for any positive numbers $X_1, X_2$. The proof is quick and uses the Fej\'er kernel to lower-bound the left hand side. One then wonders if other kernels could be employed to obtain information about $\widehat{\mu}$ on sectors or annuli. Not much progress has yet been made in this direction.
\end{remark}

Fortunately, using techniques based  \cite{BT}, we can avoid using the amplification technique entirely and obtain a lower bound of $\Omega(N^{1/5})$ for the directional discrepancy on any restricted interval of rotations. First let us briefly summarize  the results of \cite{BT} which deal with a  different, but somewhat related notion of discrepancy.  

\subsection{Discrepancy with respect to rotations and dilations of a convex body}\label{s:BT}

The paper \cite{BT} studies the problem of  discrepancy with respect to  translations and dilations (but not rotations!) of a given planar convex body with minimal smoothness and curvature assumptions (or in absence thereof). The authors use the following regularity measure:

\begin{definition}  Let $C \subset \mathbb{T}^2$
be a convex body. For every unit vector $\Theta = (\cos \theta, \sin\theta)$ and $\delta > 0$, let
$$\gamma_{\Theta}(\delta) = \Big\{x \in C: x \cdot \Theta = \inf_{y \in C} (y \cdot \Theta) + \delta\Big\},$$
in other words $\gamma_{\Theta}(\delta)$ is the chord in $C$ perpendicular to the vector $\Theta$ and a distance $\delta$ from $\partial C$. Denote by $|\gamma_{\Theta}(\delta)|$  its length. This quantity measures smoothness and convexity of the boundary $\partial C$ in the direction $\Theta$. 
\end{definition}

It is shown in \cite{BT} via a simple geometric argument that for any convex body $C$ there exists a $\delta_0$ and $c > 0$ such that for $0 < \delta < \delta_0$ and every direction $\Theta$, we have
\begin{align}
\label{eqn:gammaquantity}
 |\gamma_{\Theta}(\delta)| \ge c \delta.
 \end{align}
 This observation together with Theorem 24 in \cite{BT} provides an alternative way to prove and generalize part \eqref{t2} of Proposition \ref{p:fourier}.  
 
 We also note that if $C$ has $C^2$ boundary (e.g., a disc), then again there is a $\delta_0$ and $c> 0$ such that for $0 < \delta < \delta_0$ and each $\Theta$,
 \begin{align}
\label{eqn:gammaquantitydisc}
 |\gamma_{\Theta}(\delta)| \ge c\delta^{1/2}.
 \end{align}
The main result of \cite{BT} on the discrepancy bound is stated below. Essentially, it states that if the boundary of a convex body $C$ behaves like a disc on some interval of directions (in the sense that $\gamma_{\theta}(\delta)$ grows at least like $\delta^{1/2}$), and if we know the behavior of $\gamma_{\theta}(\delta)$ elsewhere on the boundary, we can obtain a lower bound on the discrepancy with respect to the class of all translations and dilations of $C$.

\begin{thm}[\cite{BT}, Theorem 7]
\label{thm:convexbody}
Let $C$ be a convex body and let $\Theta = (\cos \omega, \sin \omega)$. Assume there are constants $\delta_0, c_1, c_2 > 0$, $1/2 \le \sigma \le 1$ and an interval $\Omega$ in $(-\pi, \pi)$ such that for every $0 < \delta \le \delta_0$ we have

\begin{center}
$\begin{cases}
|\gamma_{-\Theta}(\delta)|  + |\gamma_{\Theta}(\delta)| > c_1 \delta^{1/2},
&  \omega \in \Omega \\
|\gamma_{-\Theta}(\delta)|  + |\gamma_{\Theta}(\delta)| > c_2 \delta^{\sigma}
& \omega \notin \Omega.
\end{cases}$
\end{center}
Then there exists $c > 0$ such that for every set $P_N$ of $N$ points in $\mathbb{T}^2$ we have
$$D_2(P_N, \mathcal{T}(C)) \ge cN^{2/(2\sigma + 3)},$$
where $\mathcal{T}(C) = \{\lambda C + t: 0 \le \lambda \le 1, t \in \mathbb{T}^2\}$.
\end{thm}

Due to \eqref{eqn:gammaquantity}, Theorem \ref{thm:convexbody} automatically gives, for a large class of convex bodies $C$, a discrepancy lower bound of order at least $N^{1/5}$ with respect to the class $\mathcal{T}(C)$. In this setting of discrepancy with respect to classes of convex bodies, we can think of the ``rotations'' component as being encoded in the smoothness of the boundary $\partial C$ on the interval $\Omega$. Thus, this setting is at least heuristically similar to the class of all translations, dilations and rotations (in the interval $\Omega$) of a rectangle $R$. 

\subsection{Proof of Theorem \ref{thm:restrictedinterval}}\label{sec:proof}

The proof of Theorem \ref{thm:convexbody} exploits the decomposition \eqref{eqn:decomposition} of the $L_2$ discrepancy that was used in the proof of Theorem \ref{thm:allrotationslowerbound}. We also make use of  an analogue this decomposition in our proof of Theorem \ref{thm:restrictedinterval}. 

\begin{remark}
\label{rmk:transfertotorus}
As described in Remark \ref{rem:discr}, in the proof of Theorem \ref{thm:restrictedinterval}, similarly to Theorem \ref{thm:convexbody}, we deal with the periodic version of discrepancy, i.e.  we transfer the problem to the torus $\mathbb T^2$ and the discrete Fourier space $\mathbb Z^2$. Obviously, the Fourier coefficients of indicators of squares are equal to their Fourier transforms evaluated at integer points. Hence, decay estimates given in Corollary \ref{cor:phir} continue to hold in this case. 
\end{remark}

The following fact, based on a classical  lemma of Montgomery  \cite{Montgomery}, provides the lower bound 
for  the ``point component''  in the proof . 
\begin{lemma}[\cite{BT}, Lemma 25]
\label{lem:montgomerylemma}
Let $B$ be a neighborhood of the origin. Then there is a positive
constant $c$ such that for every convex symmetric body $U$ in $\mathbb{R}^2$ and every finite set $\{p_j\}_{j=1}^N \subset \mathbb{T}^2$ we have
$$\sum_{m \in (U \setminus B) \cap \mathbb{Z}^2} \Big|\sum_{j=1}^N e^{2\pi i m \cdot p_j}\Big|^2 \ge N\cdot \emph{\area} (U)/4 - cN^2.$$
\end{lemma}

We can now proceed with the proof of the main theorem for discrepancy on restricted intervals.

\begin{proof}[Proof of Theorem \ref{thm:restrictedinterval}]
Let $P = \{p_1, \cdots, p_N\}$ be an $N$-point set in $[0,1]^2$, 
and let  $\Omega = [-\theta, \theta]$ with $\theta <\pi/4$ be the interval of allowed directions. 

Consider  the axis-parallel rectangle $R_0$ with vertices $(\pm \frac{X}{2}, \pm \frac{Y}{2})$, satisfying the conditions 

\begin{enumerate}
\item \label{cond:1} $XY =  \kappa N$ for some constant $\kappa$,
\item \label{cond:2} $\theta X \gtrsim Y$, and
\item \label{cond:3} $Y \gtrsim \frac{1}{\theta}$.
\end{enumerate}

Furthermore, let $$\psi = \frac{Y}{X} \hspace{0.5cm} \text{ and } \hspace{0.5cm} M = M_{\theta} = \left[ \frac{\theta}{2 \psi}\right] = \left[\frac{\theta}{2} \frac{X}{Y}\right].$$
For every $-M \le j \le M$ we consider (in the Fourier space) the rotated rectangles $R_j = r_{j\psi}R_0$, where $r_{j\psi}$ is the rotation by angle $j\psi$ about the origin. The union $\bigcup_{j=-M}^M R_j$ then approximates the sector in the plane between angles $-\theta/2$ and $\theta/2$, as depicted in Figure \ref{fig:rotatedrects} (recall that this is the sector on which $\varphi_{R,\theta} (\xi)$ satisfies a better Fourier decay estimate, see Corollary \ref{cor:phir}).

\begin{figure}[H]
\centering
\includegraphics[scale=0.25]{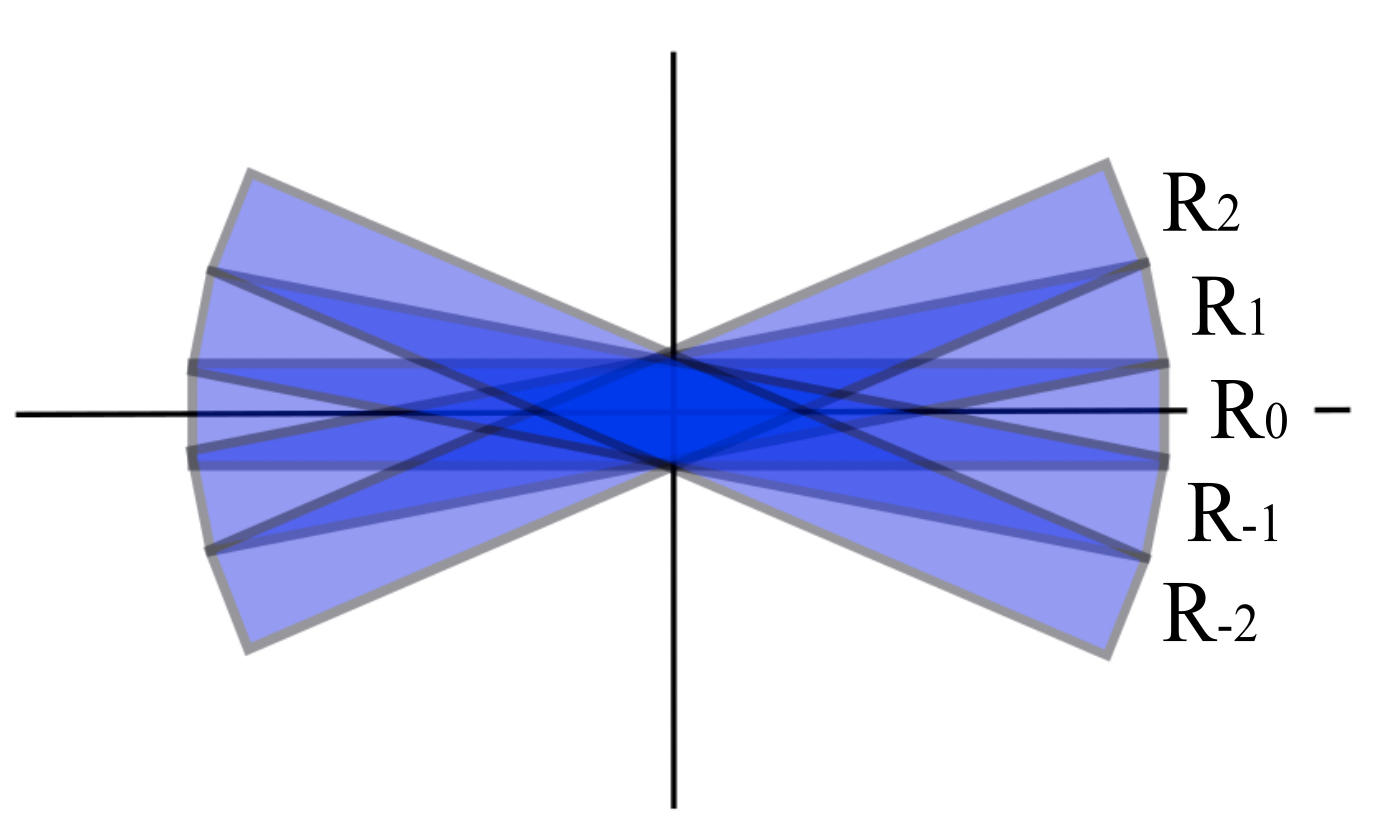}
\caption{The approximating rectangles in Fourier space when $M = 2$.}
\label{fig:rotatedrects}
\end{figure}

For every $m = (m_1, m_2) \in \mathbb{Z}^2$, let $$\Phi(m) = \sum_{j=-M}^M \fatone_{R_j}(m),$$
so that $\Phi(m)$ keeps track of how many rectangles the lattice point $m$ lies in. A simple counting argument  
shows that if $|m| \ge Y$ and $m \in \bigcup_{j=-M}^{M} R_j$,
\begin{align}
\label{eqn:phibiggerY}
\Phi(m) \lesssim \frac{Y}{|m| \psi } \approx \frac{X}{|m|},
\end{align}
and if $|m| \le Y$, we can crudely bound $\Phi(m)$ by the total number of rectangles $2M+1$, so
\begin{align}
\label{eqn:philessY}
\Phi(m) \lesssim \frac{\theta X}{Y},
\end{align}
where we have used the fact that that $2M + 1 \approx \theta \frac{X}{Y}$ since $M \approx \theta \frac{X}{Y} \gtrsim 1$ by the assumption that $\theta X \gtrsim Y$.\\

Our  goal now  is to find a constant  $\rho$ (depending on $N$ and $\theta$) so that 
\begin{equation}
\rho \Phi(m) \le \varphi_{R,\theta} (m)
\end{equation}
with some fixed $R \approx 1$, say, $R=\frac1{16}$ for $|m|\gtrsim 1$.  According to Corollary \ref{cor:phir}, it would suffice to have 
\begin{align}
\label{eqn:rhophidecay}
\rho \Phi(m) \lesssim \begin{cases} \theta^{-1}|m|^{-3} & \textup{ if } -\theta/2 \le \text{arg}(m) < \theta/2 \textup{ and } |m| \gtrsim \theta^{-1}  \\ |m|^{-4} & \textup{ otherwise.}  \end{cases}
\end{align}

\bigskip

When $m \in \cup_j R_j$ satisfies $|m| \ge Y \gtrsim \frac{1}{\theta}$, it would suffice to take  
\begin{align}
\label{eqn:rhoestimate1}
\rho \lesssim \frac{1}{\theta X^3}.
\end{align}
Indeed, in this case, since $m \in \cup_j R_j$ implies that $|M| \le X$, we see that 
$$\rho \cdot c\frac{X}{|m|} \lesssim \theta^{-1} X^{-2} |m|^{-1} \lesssim \theta^{-1} |m|^{-3},$$
which matches \eqref{eqn:phibiggerY}. \\


When $|m| \le Y$, one could take $\rho$ to satisfy 
\begin{align}
\label{eqn:rhoestimate2}
\rho \lesssim \frac{1}{\theta X  Y^3}.
\end{align}
In this case one obtains
$$\rho \cdot c\frac{\theta X}{Y} \lesssim Y^{-4} \lesssim |m|^{-4}.$$
in agreement with  \eqref{eqn:philessY}


Combining \eqref{eqn:rhoestimate1} and \eqref{eqn:rhoestimate2}, we choose
\begin{align}
\label{eqn:finalrhoestimate}
\rho = c \min\Big(\frac{1}{\theta X^3}, \frac{1}{\theta XY^3}\Big),
\end{align}
again for $c$ independent of $\theta$. \\

We are now ready to choose the parameters which would satisfy conditions \eqref{cond:1}--\eqref{cond:3}. 
Since $XY = \kappa N$ we obtain $X \approx \kappa^{3/5} N^{3/5}$ and $Y = \kappa^{2/5} N^{2/5}$ by setting equal the arguments of the min function in \eqref{eqn:finalrhoestimate}. Therefore, $\rho \approx \theta^{-1} \kappa^{-9/5} N^{-9/5}$.

Condition \eqref{cond:2} that $\theta \frac{X}{Y} \gtrsim  1$ then requires that  $\theta N^{1/5} \gtrsim 1$, so the remainder of the proof holds under the assumption that  $N \gtrsim  \theta^{-5}$. Note that then condition \eqref{cond:3} holds since  $Y \approx N^{2/5} \gtrsim \frac{1}{\theta^2} \gtrsim \frac{1}{\theta}.$

By the above discussion, for any lattice point $m$ with $|m|$ large enough (greater than an absolute constant), we have 
\begin{equation}
\label{eqn:sumofrectslower}
\varphi_{R,\theta} (m) =\frac{2}{R} \int_{R/2}^{R} \frac{1}{|\Omega|} \int_{\Omega} \int_{\mathbb{R}^2} |\widehat{\fatone}_{r,\nu}(m) |^2 d\nu dr \ge \rho \Phi(m) =  \rho  \sum_{j=-M}^M \fatone_{R_j}(m).
\end{equation}

We proceed in a fashion similar to the proof of Theorem \ref{thm:allrotationslowerbound}. Defining  the discrepancy measure $\mu  = \sum_{i=1}^N \delta_{p_i} - N\lambda$ where $\lambda$ is the Lebesgue measure on the torus $\mathbb T^2$,  considering an analog of \eqref{eqn:decomposition} for the $L^2$ discrepancy with respect to translations, and averaging over dilations and rotations, one obtains 
\begin{align}
\begin{split}
\label{eqn:discrep_calc_restricted}
\frac{2}{R} \int_{R/2}^R \frac{1}{\Omega} \int_{\Omega}  \|D(P_N,R(q,r,\nu)\|_{L^2(dq)} ^2 d\nu dr & = \frac{2}{R} \int_{R/2}^{R} \frac{1}{|\Omega|} \int_{\Omega} \sum_{m \neq 0} |\widehat{\fatone}_{r,\nu}(m) |^2 |\widehat{\mu}(m)|^2 d\nu dr\\
&= \sum_{m \neq 0} \Big|\sum_{j=1}^N e^{2\pi i m p_j} \Big| \frac{2}{R} \int_{R/2}^R \frac{1}{|\Omega|} \int_{\Omega} |\widehat{\fatone}_{r,\nu}(m)|^2 d\nu dr\\
&= \sum_{m \neq 0} \Big|\sum_{j=1}^N e^{2\pi i m p_j} \Big| \cdot  \varphi_{R,\theta} (m) \\
&\ge  \sum_{|m| \ge K} \Big|\sum_{j=1}^N e^{2\pi i m p_j}\Big|^2 \rho \sum_{j=-M}^M \fatone_{R_j}(m)  \hspace{0.5cm}  \text{ (by \ref{eqn:sumofrectslower})}\\
&= \rho \sum_{j=-M}^M \sum_{\substack{|m| \ge K, \\  m \in R_j}} \Big|\sum_{j=1}^N e^{2\pi i m p_j}\Big|^2.\\
\end{split}
\end{align}
In the computation above,  $K$ is  an absolute constant large enough so that the decay estimates in Corollary \ref{cor:phir} apply. 
Using Lemma \ref{lem:montgomerylemma} and taking $\kappa$ large enough (depending only on  $c$ from Lemma \ref{lem:montgomerylemma}, but independent of $N$ and $\theta$), the expression above can be bounded below by
\begin{align*}
\rho  \sum_{j=-M_{\theta}}^{M_{\theta}} (N\cdot \text{area}(R_j)/4 - cN^2)
& \ge \rho (2M_{\theta}+1)\left(\frac14 \kappa N^2-cN^2\right)\\
& \approx  \theta^{-1}  N^{-9/5} \cdot \frac{\theta X}{Y} \cdot N^2  \approx N^{-\frac95 + \frac35-\frac25 + 2} = N^{\frac25}.
\end{align*}
Taking square roots, we conclude that for $N \gtrsim \theta^{-5}$,
$$D(N, \mathcal{R}_{\Omega}) \gtrsim c N^{1/5},$$
where $c$ is independent of $\theta$.

\end{proof}

\begin{proof}[Proof of Corollary \ref{cor:thm}]
Theorem \ref{thm:restrictedinterval} automatically implies $D(N, \mathcal{R}_{\Omega}) \ge c\theta N^{1/5}$ for $N \ge  \gamma \theta^{-5}$. 
As in \eqref{eqn:trivialbound} in the proof of Theorem \ref{thm:allrotationslowerbound}, it is easy to see that for small rectangles of area $\frac{1}{4N}$, the discrepancy is at least $\frac{1}{4}$. Thus,  
when $N \le \gamma \theta^{-5}$, we trivially have $D(N, \mathcal{R}_{\theta}) \ge \frac{1}{4} \ge c \gamma^{-1/5} \theta N^{1/5}.$
\end{proof}

It is still unclear whether this is the best lower bound for the discrepancy with respect to restricted intervals. Indeed, in the setting of translations and dilations of a convex set, \cite{BT} prove that the bound is sharp by finding an explicit point set that exhibits an upper bound of order $N^{1/5}$; however, the proof techniques do not seem to translate in this setting. It is natural to expect that one should be able to improve the bound to order $N^{1/4}$, as in the case of all rotations. Diophantine properties play a large role in discrepancy (see \cite{BMPS11,BMPS16}), and 
since in any interval one can find a number of any Diophantine type, it seems on a heuristic level that the case of all rotations and the case of a restricted interval should not differ much with respect to the discrepancy. There remain many open problems in the area of directional discrepancy, and any further understanding of the relationship between the allowed rotation set $\Omega$ and the discrepancy of $\mathcal{R}_{\Omega}$ would be very interesting.

\end{document}